\documentclass[10pt]{amsart}
\usepackage{amsmath,amssymb,amsthm,mathrsfs,graphicx,mathtools}
\usepackage[margin=1.5in]{geometry}
\usepackage[colorlinks=true,linkcolor=blue,citecolor=blue]{hyperref}
\newtheorem{theorem}{Theorem}[section]
\newtheorem{lemma}[theorem]{Lemma}
\newtheorem{proposition}[theorem]{Proposition}
\newtheorem{corollary}[theorem]{Corollary}
\theoremstyle{definition}

\theoremstyle{remark}

\numberwithin{equation}{section}

\DeclareMathOperator{\id}{id}
\DeclareMathOperator{\tr}{tr}
\allowdisplaybreaks[4]

\begin{document}

\title{Two Kinds of Curvature Operators in Dimension Four}

\author[Xiaodong Cao]{Xiaodong Cao}
\address{Department of Mathematics, Cornell University, Ithaca, NY, 14853}
\email{xiaodongcao@cornell.edu}

\author[Xiaolong Li]{Xiaolong Li}
\address{Department of Mathematics and Statistics, Auburn University, Auburn, AL, 36849, USA}
\email{xil0005@auburn.edu}
\thanks{The second author was supported in part by NSF-DMS \#2553660 and a start-up grant at Auburn University.}



\begin{abstract}
We prove an invariant identity relating the curvature operators of the
first and second kind in dimension four. Under the natural identification
$S^2_0\cong\Lambda^+\otimes\Lambda^-$, the trace-free Ricci tensor
contributes through a single commutator term. As a consequence, we relate
the curvature operator of the second kind on products of unit self-dual
and anti-self-dual two-forms to biorthogonal sectional curvature.
We then classify closed oriented four-manifolds satisfying
$K^\perp_{\max}\leq S/6$, where $K^\perp_{\max}$ denotes the maximum
biorthogonal sectional curvature and $S$ is the scalar curvature.
The classification includes the cases in which the scalar curvature vanishes.
\end{abstract}
\maketitle

\section{Introduction}\label{sec:introduction}

The Riemannian curvature tensor $R$ of a Riemannian manifold $(M^n,g)$
induces two self-adjoint curvature operators: the curvature operator of
the first kind $\widehat R:\Lambda^2TM\to\Lambda^2TM$ acting on two-forms and the curvature
operator of the second kind $\mathring R:S^2_0TM\to S^2_0TM$ acting on traceless symmetric two-tensors.
Their definitions and our conventions are given in
Section~\ref{sec:preliminaries}. Positivity conditions on these operators
have strong geometric consequences. For instance, B\"ohm and Wilking~\cite{BW08}
proved that a closed Riemannian manifold with two-positive curvature
operator of the first kind is diffeomorphic to a spherical space form.
For the curvature operator of the second kind, Nishikawa~\cite{Nishikawa86}
conjectured the same conclusion under positivity. Cao, Gursky, and
Tran~\cite{CGT23} proved the stronger statement under two-positivity.
Together with the work of the second
author~\cite{Li24} and Nienhaus, Petersen, and Wink~\cite{NPW22}, these
results show that a closed Riemannian manifold with three-nonnegative
curvature operator of the second kind is either flat or diffeomorphic
to a spherical space form. Here $k$-nonnegativity means that the sum of
the $k$ smallest eigenvalues is nonnegative, and $k$-positivity means
that this sum is positive. For nonintegral $\alpha$, the corresponding
sum is defined by linear interpolation between consecutive integer values.
The second author obtained sphere theorems
under $4\frac12$-positivity~\cite{Li22JGA} and classification results
for K\"ahler surfaces under six-nonnegativity~\cite{Li22PAMS}.
For K\"ahler manifolds of arbitrary complex dimension, he obtained
further rigidity and classification results~\cite{Li23MZ}.
We refer the reader to
\cite{Li22product,Li25JLMS,NPWW22,FL24,DF24,CW25,CW26,LP26}
for further developments.

In dimension three, the Weyl tensor vanishes, and Fluck and the second
author~\cite{FL24} expressed the eigenvalues of $\mathring R$ explicitly
in terms of those of $\widehat R$. Using these formulas, they proved
that $\alpha$-positivity and $\alpha$-nonnegativity of the curvature
operator of the second kind are preserved by the Ricci flow on closed
three-manifolds for every $\alpha\in[1,5]$.
In dimension four, we use the Hodge decomposition
$\Lambda^2V=\Lambda^+V\oplus\Lambda^-V$ to relate the two operators.
We identify two-forms
with skew-adjoint endomorphisms and write $\omega\eta$ for their
composition. If $\omega\in\Lambda^+V$ and $\eta\in\Lambda^-V$, this
composition belongs to $S^2_0(V)$. Our main result is the following
identity.

\begin{theorem}\label{thm:main}
Let $(V,g)$ be an oriented Euclidean vector space of dimension four,
and let $R$ be an algebraic curvature tensor on $V$. Denote its scalar
curvature by $S$ and its curvature operators of the first and second
kind by $\widehat R$ and $\mathring R$, respectively. Then
\begin{equation}\label{eq:master}
 \mathring R(\omega\eta,\xi\zeta)
 +\langle\eta,\zeta\rangle\widehat R(\omega,\xi)
 +\langle\omega,\xi\rangle\widehat R(\eta,\zeta)
 +\frac12\widehat R([\omega,\xi],[\eta,\zeta])
 =\frac S4\langle\omega,\xi\rangle\langle\eta,\zeta\rangle
\end{equation}
for all $\omega,\xi\in\Lambda^+V$ and $\eta,\zeta\in\Lambda^-V$.
Here $\Lambda^\pm V$ are the $\pm1$ eigenspaces of the Hodge star,
and $[\cdot,\cdot]$ is the Lie bracket under the identification
$\Lambda^2V\cong\mathfrak{so}(4)$.
\end{theorem}

In particular, if $\{\omega_i\}_{i=1}^3$ and $\{\eta_j\}_{j=1}^3$
are orthonormal bases of $\Lambda^+V$ and $\Lambda^-V$, then
\begin{equation}\label{eq:main identity basis}
 \mathring R(\omega_i\eta_j,\omega_k\eta_l)
 +\delta_{jl}\widehat R(\omega_i,\omega_k)
 +\delta_{ik}\widehat R(\eta_j,\eta_l)
 +\frac12\widehat R([\omega_i,\omega_k],[\eta_j,\eta_l])
 =\frac S4\delta_{ik}\delta_{jl}.
\end{equation}
Cao, Gursky, and Tran~\cite{CGT23} obtained a $9\times9$ matrix normal
form for the curvature operator of the second kind by choosing
eigenbases of the Weyl operators $W^+$ and $W^-$. In this normal form,
the diagonal blocks are determined by the Weyl tensor and the scalar
curvature, and the skew-symmetric off-diagonal blocks are determined
by the trace-free Ricci tensor. Choosing these eigenbases in
\eqref{eq:main identity basis} recovers their normal form.
Theorem~\ref{thm:main} gives an invariant formulation in which the
trace-free Ricci contribution is expressed by the commutator term.

Setting $\xi=\omega$ and $\zeta=\eta$ gives the following consequence.
\begin{corollary}\label{cor:diagonal-identity}
For $\omega\in\Lambda^+V$ and $\eta\in\Lambda^-V$,
\begin{equation}\label{eq:diagonal-identity}
 \mathring R(\omega\eta,\omega\eta)
 +|\eta|^2\widehat R(\omega,\omega)
 +|\omega|^2\widehat R(\eta,\eta)
 =\frac S4|\omega|^2|\eta|^2.
\end{equation}
\end{corollary}

To formulate the geometric application, define
\begin{equation}\label{eq:Sigma-definition}
 \Sigma_p=\left\{h\in S^2_0(T_pM):h^2=\frac14\id\right\}.
\end{equation}
Here we view symmetric tensors as self-adjoint endomorphisms.
Every $h\in\Sigma_p$ has $|h|=1$, and
\[
 \Sigma_p=\{\omega\eta: \omega\in\Lambda^+T_pM,\
 \eta\in\Lambda^-T_pM,\ |\omega|=|\eta|=1\}.
\]
We write $\mathring R|_\Sigma\geq\kappa$ if
$\mathring R(h,h)\geq\kappa(p)$ for every $p\in M$ and every
$h\in\Sigma_p$, and use the analogous notation for strict inequalities.
This restriction has a direct interpretation in terms of
\emph{biorthogonal sectional curvature}. For a two-plane
$P\subset T_pM$, set
\[
 K^\perp(P)=\frac{K(P)+K(P^\perp)}2,
 \qquad
 K^\perp_{\max}(p)=\max_{P\subset T_pM}K^\perp(P).
\]
If $h_P=\tfrac12(\pi_{P^\perp}-\pi_P)$, where $\pi_P$ is orthogonal
projection onto $P$, Corollary~\ref{cor:diagonal-identity} gives
\begin{equation*}
 \mathring R(h_P,h_P)=\frac S4-2K^\perp(P).
\end{equation*}
Every element of $\Sigma_p$ occurs in this way. Consequently,
\begin{equation*}
 \mathring R|_\Sigma\geq-\frac S{12}
 \quad\Longleftrightarrow\quad
 K^\perp_{\max}\leq\frac S6,
\end{equation*}
with the corresponding equivalence for strict inequalities.
Thus this condition depends only on the Weyl tensor and the scalar
curvature, whereas $\mathring R$ also depends on the trace-free Ricci tensor.
On an Einstein four-manifold, biorthogonal sectional curvature agrees
with sectional curvature. For related classification results under
sectional curvature pinching, see Cao and Tran~\cite{CT18}.

Biorthogonal curvature has been used to study the topology and
geometry of four-manifolds; see, in particular,
Costa and Ribeiro~\cite{CR14}. Under the assumption $S>0$, they proved
a classification under the condition
$K^\perp_{\min}\geq S/24$ and observed that the same classification
holds under $K^\perp_{\max}\leq S/6$.
The case $S>0$ of the next theorem therefore has a known
biorthogonal formulation. We express the condition in terms of the
curvature operator of the second kind and include the cases in which
the scalar curvature vanishes. The proof uses the classification
of four-manifolds with positive isotropic curvature by
Chen, Tang, and Zhu~\cite{CTZ12}, which builds on Hamilton's fundamental
curvature pinching estimates and Ricci flow with surgery~\cite{Ham97}.
We also use the Ricci flow analysis of
Micallef and Wang~\cite{MW93} and the rigidity of compact self-dual
K\"ahler surfaces~\cite{Der83,Ito84}. A related rigidity theorem of
Hitchin~\cite{Hit81} states that a compact self-dual Einstein
four-manifold with positive scalar curvature is, up to scaling,
the round $\mathbb{S}^4$ or $\mathbb{CP}^2$ with the Fubini-Study metric.

\begin{theorem}\label{thm:pinching}
Let $(M^4,g)$ be a closed, connected, oriented Riemannian four-manifold.
\begin{enumerate}
\item If $\mathring R|_\Sigma>-S/12$, then $(M,g)$ has positive
isotropic curvature. Moreover, $M$ is diffeomorphic to $\mathbb{S}^4$ or to a
finite connected sum
\[
 \mathop{\#}_{r=1}^{N}(\mathbb{S}^3\times\mathbb R)/\Gamma_r,
\]
where each $\Gamma_r$ is a cocompact discrete subgroup of
$\operatorname{Isom}(\mathbb{S}^3\times\mathbb R)$ acting freely, and each
quotient is orientable. In particular, if $M$ is simply connected,
then $M$ is diffeomorphic to $\mathbb{S}^4$.
\item If $\mathring R|_\Sigma\geq-S/12$, then $(M,g)$ has
nonnegative isotropic curvature, and exactly one of the following
holds:
\begin{enumerate}
\renewcommand{\theenumii}{\alph{enumii}}
\item $M$ admits a metric with positive isotropic curvature and has
one of the diffeomorphism types in part~\textup{(1)};
\item $(M,g)$ is flat;
\item the universal Riemannian cover is isometric to
$\mathbb{S}^2(c)\times\mathbb H^2(-c)$ for some $c>0$;
\item $(M,g)$ is isometric, up to scaling, to $\mathbb{CP}^2$
with the Fubini-Study metric.
\end{enumerate}
In particular, if $M$ is simply connected, then $M$ is diffeomorphic
to $\mathbb{S}^4$ or $\mathbb{CP}^2$.
\end{enumerate}
\end{theorem}

Equality is attained by the Fubini-Study metric and by
products $\mathbb{S}^2(c)\times\mathbb H^2(-c)$; closed examples of the latter
type are obtained by taking a closed hyperbolic quotient of the second
factor. Together with flat metrics, these examples realize
parts~(2)(b)-(2)(d) of
Theorem~\ref{thm:pinching}.

The paper is organized as follows. Section~\ref{sec:preliminaries}
fixes the conventions for the two curvature operators.
Section~\ref{sec:dimension-four} describes the Hodge decomposition,
the relation between the two operators in dimension four, and
biorthogonal curvature.
Section~\ref{sec:main-proof} proves Theorem~\ref{thm:main} and
Corollary~\ref{cor:diagonal-identity}, including the full matrix
calculation. Section~\ref{sec:pinching-proof} proves
Theorem~\ref{thm:pinching}.

\medskip
\noindent\textbf{AI disclosure.}
The identity in Theorem~\ref{thm:main} was discovered by the authors
two years before the preparation of this manuscript. AI tools were
only used to assist with the writing. The authors are responsible for the mathematical content.

\medskip
\noindent\textbf{Acknowledgments.}
This project was completed during the second author's multiple visits
to the Department of Mathematics at Cornell University. He thanks
the department for its hospitality.

\section{Preliminaries}\label{sec:preliminaries}

Let $(V,g)$ be a Euclidean vector space of dimension $n\geq2$, and let
$\{e_1,\ldots,e_n\}$ be an orthonormal basis. We identify $V$ with
$V^*$ using $g$. Denote by $\Lambda^2V$, $S^2(V)$, and $S^2_0(V)$
the spaces of two-forms, symmetric two-tensors, and traceless symmetric
two-tensors, respectively. Their inner products are given by
\begin{equation*}
 \langle \alpha,\beta\rangle_{\Lambda^2}
 =\frac12\sum_{i,j=1}^n\alpha_{ij}\beta_{ij},
 \qquad
 \langle h,k\rangle_{S^2}
 =\sum_{i,j=1}^n h_{ij}k_{ij}.
\end{equation*}
We use the wedge product and symmetric product
\[
 u\wedge v=u\otimes v-v\otimes u,
 \qquad
 u\odot v=u\otimes v+v\otimes u.
\]
Then $\{e_i\wedge e_j:1\leq i<j\leq n\}$ is an orthonormal basis of
$\Lambda^2V$, and
\[
 \left\{\frac1{\sqrt2}e_i\odot e_j:1\leq i<j\leq n\right\}
 \cup\left\{\frac12e_i\odot e_i:1\leq i\leq n\right\}
\]
is an orthonormal basis of $S^2(V)$.

An algebraic curvature tensor on $V$ is a multilinear map
$R:V^4\to\mathbb R$ satisfying the symmetries
\[R(X,Y,Z,W)=-R(Y,X,Z,W)=-R(X,Y,W,Z)=R(Z,W,X,Y)\]
and the first Bianchi identity
\[R(X,Y,Z,W)+R(Y,Z,X,W)+R(Z,X,Y,W)=0\]
for all $X,Y,Z,W\in V$.
We write $R_{ijkl}=R(e_i,e_j,e_k,e_l)$ and adopt the sign convention for
which $R_{ijij}>0$ for $i\neq j$ on the round sphere. The Ricci tensor
and scalar curvature are
\[
 R_{ij}=\sum_{k=1}^n R_{kikj},
 \qquad
 S=\sum_{i=1}^n R_{ii}
  =\sum_{i,j=1}^n R_{ijij}.
\]
We say that $R$ is Einstein if $R_{ij}=\lambda g_{ij}$ for some
$\lambda\in\mathbb R$.

The curvature operator of the first kind is the self-adjoint operator
$\widehat R:\Lambda^2V\to\Lambda^2V$ defined by
\[
 \widehat R(e_i\wedge e_j)=\frac12\sum_{k,l}R_{ijkl}\,e_k\wedge e_l. 
\]
Equivalently,
\[
 \langle\widehat R(X\wedge Y),Z\wedge W \rangle
 =R(X,Y,Z,W).
\]
We use $\widehat R(\alpha,\beta)$ for $\langle\widehat R\alpha,\beta\rangle$.
The curvature operator of the second kind is
\[
 \mathring R=\pi\circ\left.\overline R\right|_{S^2_0(V)}:
 S^2_0(V)\to S^2_0(V),
\]
where $\overline R:S^2(V)\to S^2(V)$ is defined by
\[
 \overline R(e_i\odot e_j)=\sum_{k,l}R_{iklj}\,e_k\odot e_l,
\]
and $\pi(h)=h-\frac{\tr(h)}n g$ is the orthogonal projection onto
$S^2_0(V)$. Both $\overline R$ and $\mathring R$ are self-adjoint.
For $h,k\in S^2_0(V)$, we write
\[
 \mathring R(h,k)
 =\langle\mathring Rh,k\rangle
 =\langle\overline Rh,k\rangle.
\]
Since $\overline R$ is self-adjoint and
$\overline R(g)=-\operatorname{Ric}$, it preserves $S^2_0(V)$ if and
only if $R$ is Einstein. If $R$ has constant sectional curvature
$\kappa$, then $\widehat R=\kappa\id_{\Lambda^2V}$ and
$\mathring R=\kappa\id_{S^2_0(V)}$.

For later computations, we recall the following identities
from~\cite[Proposition~2.6]{Li25JLMS}.
\begin{lemma}\label{lemma: eq identities}
For all indices $i,j,k,l$, we have
\begin{equation*}
 \langle\overline R(e_i\odot e_j),e_k\odot e_l\rangle
 =2\bigl(R_{iklj}+R_{ilkj}\bigr).
\end{equation*}
If $i\ne j$ and $k\ne l$, then
\begin{equation*}
 \langle \mathring R(e_i\odot e_j),e_k\odot e_l\rangle
 =2\bigl(R_{iklj}+R_{ilkj}\bigr).
\end{equation*}
\end{lemma}
\begin{proof}
The definition of the symmetric product gives
\[
 \langle e_a\odot e_b,e_k\odot e_l\rangle
 =2(\delta_{ak}\delta_{bl}+\delta_{al}\delta_{bk}).
\]
Applying this to the definition of $\overline R$ proves the first
identity. When $i\ne j$ and $k\ne l$, both symmetric products are
trace-free, so orthogonal projection onto $S^2_0(V)$ does not change
the pairing. This proves the second identity.
\end{proof}

\section{Curvature operators in dimension four}\label{sec:dimension-four}

Throughout this section, let $(V,g)$ be an oriented Euclidean vector
space of dimension four. The Hodge star operator gives the orthogonal
decomposition
\begin{equation*}
    \Lambda^2 V =\Lambda^+ V \oplus \Lambda^- V,
\end{equation*}
where $\Lambda^+V$ and $\Lambda^-V$ are its eigenspaces with
eigenvalues $1$ and $-1$, respectively.
For an oriented orthonormal basis \(\{e_1,\dots,e_4\}\), set
\begin{equation}\label{eq: omega def}
\begin{aligned}
\omega_1 &= \tfrac1{\sqrt2}(e_1\wedge e_2+e_3\wedge e_4),\\
\omega_2 &= \tfrac1{\sqrt2}(e_1\wedge e_3+e_4\wedge e_2),\\
\omega_3 &= \tfrac1{\sqrt2}(e_1\wedge e_4+e_2\wedge e_3),
\end{aligned}
\end{equation}
and 
\begin{equation}\label{eq: eta def}
\begin{aligned}
\eta_1 &= \tfrac1{\sqrt2}(e_1\wedge e_2-e_3\wedge e_4),\\
\eta_2 &= \tfrac1{\sqrt2}(e_1\wedge e_3-e_4\wedge e_2),\\
\eta_3 &= \tfrac1{\sqrt2}(e_1\wedge e_4-e_2\wedge e_3).
\end{aligned}
\end{equation}
Then $\{\omega_1,\omega_2,\omega_3\}$ and
$\{\eta_1,\eta_2,\eta_3\}$ are oriented orthonormal bases of
$\Lambda^+V$ and $\Lambda^-V$, respectively.

Relative to the decomposition \(\Lambda^2 V = \Lambda^+V \oplus \Lambda^-V\), we have
\begin{equation*}
 \widehat R
 =\begin{pmatrix}A&B\\ B^t&C\end{pmatrix}
 =\begin{pmatrix}
 W^++\dfrac{S}{12}\id_{\Lambda^+}&B\\[3pt]
 B^t&W^-+\dfrac{S}{12}\id_{\Lambda^-}
 \end{pmatrix}.
\end{equation*}
Here $W^\pm$ are the self-dual and anti-self-dual Weyl operators, and
\begin{equation*}
 \tr A=\tr C=\frac S4.
\end{equation*}
The off-diagonal block $B:\Lambda^-V\to\Lambda^+V$ is determined
by the trace-free Ricci tensor. In particular, $B=0$ if and only
if $R$ is Einstein.

We next describe the isometry between $\Lambda^+V\otimes\Lambda^-V$
and $S^2_0(V)$. Identify a two-form $\alpha$ with the skew-adjoint
endomorphism $\alpha^\sharp$ defined by
\begin{equation*}
 g(\alpha^\sharp x,y)=\alpha(x,y).
\end{equation*}
If $\omega\in\Lambda^+V$ and $\eta\in\Lambda^-V$, then
$\omega^\sharp$ and $\eta^\sharp$ commute.  Their composition is symmetric,
and
\[
 \tr(\omega^\sharp\eta^\sharp)
 =-2\langle\omega,\eta\rangle=0.
\]
Consequently, the map \(\Phi:\Lambda^+V\otimes\Lambda^-V\longrightarrow S^2_0(V)\)
given by
\begin{equation*}
 \Phi(\omega\otimes\eta)(x,y)
 =g(\omega^\sharp\eta^\sharp x,y),
\end{equation*}
is well defined. Moreover,
\begin{equation*}
 \langle\omega^\sharp\eta^\sharp,
        \xi^\sharp\zeta^\sharp\rangle_{S^2}
 =\langle\omega,\xi\rangle_{\Lambda^2}
  \langle\eta,\zeta\rangle_{\Lambda^2},
\end{equation*}
for $\omega,\xi\in\Lambda^+V$ and $\eta,\zeta\in\Lambda^-V$.
This follows from the anticommutator identities
\begin{align*}
     \omega^\sharp\xi^\sharp+\xi^\sharp\omega^\sharp
 & =-\langle\omega,\xi\rangle\id, \\
  \eta^\sharp\zeta^\sharp+\zeta^\sharp\eta^\sharp
 &=-\langle\eta,\zeta\rangle\id, 
\end{align*}
together with the orthogonality of $\Lambda^+V$ and $\Lambda^-V$. 
Since both spaces have dimension nine, $\Phi$ is an isometric
isomorphism. We henceforth write
$\omega\eta=\omega^\sharp\eta^\sharp$. In particular,
$\{\omega_i\eta_j:1\leq i,j\leq3\}$ is an orthonormal basis of
$S^2_0(V)$.

Under the identification $\Lambda^2V\cong\mathfrak{so}(4)$, the
Lie bracket of two-forms is defined by
\begin{equation*}
 \bigl([\alpha,\beta]\bigr)^\sharp
 =\alpha^\sharp\beta^\sharp-\beta^\sharp\alpha^\sharp.
\end{equation*}
Let $\varepsilon_{ijk}$ be the alternating symbol with
$\varepsilon_{123}=1$. The bases in \eqref{eq: omega def} and
\eqref{eq: eta def} satisfy
\begin{equation}\label{eq:omega eta commutator}
 [\omega_i,\omega_k]
 =\sqrt2\,\varepsilon_{ikm}\omega_m,
 \qquad
[\eta_j,\eta_l]
 =-\sqrt2\,\varepsilon_{jln}\eta_n.    
\end{equation}
Here and below, repeated indices in tensor contractions are summed
from $1$ to $3$. The opposite signs in
\eqref{eq:omega eta commutator} reflect the two quaternionic orientations.

We now verify the description of $\Sigma_p$ in the introduction.
On $V$, write $\Sigma=\{h\in S^2_0(V):h^2=\frac14\id\}$,
as in \eqref{eq:Sigma-definition}. The anticommutator identities give
\[
 (\omega^\sharp)^2=-\frac{|\omega|^2}{2}\id,
 \qquad
 (\eta^\sharp)^2=-\frac{|\eta|^2}{2}\id.
\]
Since the two endomorphisms commute,
\[
 (\omega\eta)^2=\frac{|\omega|^2|\eta|^2}{4}\id,
 \qquad |\omega\eta|^2=|\omega|^2|\eta|^2.
\]
Thus the product of two unit forms belongs to $\Sigma$.
Conversely, $h\in\Sigma$ has eigenvalues $-1/2$ and $1/2$, each with
multiplicity two because $\tr h=0$. Let $P$ be its $-1/2$ eigenspace
and choose an oriented orthonormal basis $e_1,e_2$ of $P$, extended
to a positively oriented orthonormal basis of $V$. With
$u=e_1\wedge e_2$, define
\begin{equation}\label{eq:plane-forms}
 \omega=\frac{u+\star u}{\sqrt2},\qquad
 \eta=\frac{u-\star u}{\sqrt2}.
\end{equation}
Then $|\omega|=|\eta|=1$ and
\[
 \omega\eta=\frac12(\pi_{P^\perp}-\pi_P)=h.
\]
Thus every element of $\Sigma$ arises from a two-plane. Reversing
the orientation of $P$ changes both forms by a sign and leaves their
product unchanged; replacing $P$ by $P^\perp$ changes $h$ to $-h$.

For the forms in \eqref{eq:plane-forms}, one has
$u=(\omega+\eta)/\sqrt2$ and
$\star u=(\omega-\eta)/\sqrt2$. It follows that
\begin{align}
 K^\perp(P)
 &=\frac12\bigl(\widehat R(u,u)+\widehat R(\star u,\star u)\bigr)
 \notag\\*
 &=\frac12\bigl(\widehat R(\omega,\omega)
                  +\widehat R(\eta,\eta)\bigr)\notag\\*
 &=\frac S{12}+\frac12\bigl(\langle W^+\omega,\omega\rangle
                          +\langle W^-\eta,\eta\rangle\bigr).
 \label{eq:biorthogonal-weyl}
\end{align}
The mixed block $B$ cancels in this average. Conversely, any pair
of unit forms $\omega\in\Lambda^+$ and $\eta\in\Lambda^-$ gives a
unit decomposable two-form $u=(\omega+\eta)/\sqrt2$, since
$u\wedge u=0$. Thus the two unit forms in
\eqref{eq:biorthogonal-weyl} vary independently.

Write the eigenvalues of $W^\pm$ as
$w_1^\pm\leq w_2^\pm\leq w_3^\pm$. Since $W^\pm$ are trace-free,
$w_1^\pm+w_2^\pm+w_3^\pm=0$. Taking the minimum and maximum in
\eqref{eq:biorthogonal-weyl} gives~\cite{CR14}
\begin{equation}\label{eq:biorthogonal-extrema}
 K^\perp_{\min}=\frac S{12}+\frac{w_1^++w_1^-}{2},
 \qquad
 K^\perp_{\max}=\frac S{12}+\frac{w_3^++w_3^-}{2}.
\end{equation}

\section{Proof of Theorem~\ref{thm:main} and Corollary~\ref{cor:diagonal-identity}}\label{sec:main-proof}

We first prove the identity for the bases introduced in
Section~\ref{sec:dimension-four}.
\begin{proposition}\label{prop: main identity}
Let $\{\omega_i\}_{i=1}^3$ and $\{\eta_j\}_{j=1}^3$ be the bases
defined in \eqref{eq: omega def} and \eqref{eq: eta def}.
For all $1\leq i,j,k,l\leq3$, we have
\begin{equation}\label{eq:main identity basis s3}
 \mathring R(\omega_i\eta_j,\omega_k\eta_l)
 +\delta_{jl}\widehat R(\omega_i,\omega_k)
 +\delta_{ik}\widehat R(\eta_j,\eta_l)
 +\frac12\widehat R([\omega_i,\omega_k],[\eta_j,\eta_l])
 =\frac S4\delta_{ik}\delta_{jl}.
\end{equation}
\end{proposition}

\begin{proof}[Proof of Proposition~\ref{prop: main identity}]
With respect to the ordered orthonormal basis
$\{\omega_1,\omega_2,\omega_3,\eta_1,\eta_2,\eta_3\}$, write
\begin{equation*}
\widehat R
 =\begin{pmatrix}A&B\\ 
 B^t&C\end{pmatrix}.
\end{equation*}
Direct calculation gives the entries of $A$, $C$, and $B$:
\begin{align*}
2 A_{11} = 2\widehat R (\omega_1,\omega_1) &= R_{1212} +R_{3434} +2R_{1234}, \\
2 A_{22} = 2\widehat R (\omega_2,\omega_2) &= R_{1313} +R_{2424} +2R_{1342}, \\
2 A_{33} = 2\widehat R (\omega_3,\omega_3) &= R_{1414} +R_{2323} +2R_{1423}, \\
2 A_{12} = 2\widehat R (\omega_1,\omega_2) &= R_{1213} +R_{1242} +R_{1334}+R_{3442}, \\
2 A_{13} = 2\widehat R (\omega_1,\omega_3) &= R_{1214} +R_{1223} +R_{1434}+R_{2334}, \\*
2 A_{23} = 2\widehat R (\omega_2,\omega_3) &= R_{1314} +R_{1323} +R_{1442}+R_{2342}, 
\end{align*}
\begin{align*}
2 C_{11} = 2\widehat R (\eta_1,\eta_1) &= R_{1212} +R_{3434} -2R_{1234}, \\
2 C_{22} = 2\widehat R (\eta_2,\eta_2) &= R_{1313} +R_{2424} -2R_{1342}, \\
2 C_{33} =2\widehat R (\eta_3,\eta_3) &= R_{1414} +R_{2323} -2R_{1423}, \\
2 C_{12} = 2\widehat R (\eta_1,\eta_2) &= R_{1213} -R_{1242} -R_{1334}+R_{3442}, \\
2 C_{13} = 2\widehat R (\eta_1,\eta_3) &= R_{1214} -R_{1223} -R_{1434}+R_{2334}, \\*
2 C_{23} = 2\widehat R (\eta_2,\eta_3) &= R_{1314} -R_{1323} -R_{1442}+R_{2342}, 
\end{align*}
and
\begin{align*}
2 B_{11} = 2\widehat R (\omega_1,\eta_1) &= R_{1212}-R_{3434}, \\
2 B_{12} = 2\widehat R (\omega_1,\eta_2) &= R_{1213}-R_{1242}+R_{1334}-R_{3442}, \\
2 B_{13} = 2\widehat R (\omega_1,\eta_3) &= R_{1214}-R_{1223} +R_{1434}-R_{2334}, \\
2 B_{21} = 2\widehat R (\omega_2,\eta_1) &= R_{1213}-R_{1334}+R_{1242}-R_{3442}, \\
2 B_{22} = 2\widehat R (\omega_2,\eta_2) &= R_{1313}-R_{4242}, \\
2 B_{23} = 2\widehat R (\omega_2,\eta_3) &= R_{1314}-R_{1323}+R_{1442}-R_{2342}, \\
2 B_{31} = 2\widehat R (\omega_3,\eta_1) &= R_{1214}-R_{1434} +R_{1223} -R_{2334}, \\
2 B_{32} = 2\widehat R (\omega_3,\eta_2) &= R_{1314}-R_{1442}+R_{1323}-R_{2342}, \\*
2 B_{33} = 2\widehat R (\omega_3,\eta_3) &= R_{1414}-R_{2323}, 
\end{align*}
 
The products $\omega_i\eta_j$ give the following ordered orthonormal
basis of $S^2_0(V)$:
\begin{align*}
h_1 &= \omega_1 \eta_1 = \tfrac{1}{4}(-e_1 \odot e_1 -e_2 \odot e_2 +e_3 \odot e_3 +e_4 \odot e_4),\\   
h_2 &= \omega_1 \eta_2 =\tfrac{1}{2}(e_1\odot e_4 -e_2 \odot e_3),\\   
h_3 &= \omega_1 \eta_3 =\tfrac{1}{2}(-e_1 \odot e_3 -e_2 \odot e_4),\\   
h_4 &= \omega_2 \eta_1 =\tfrac{1}{2}(-e_1 \odot e_4 -e_2 \odot e_3),\\   
h_5 &= \omega_2 \eta_2 =\tfrac{1}{4}(-e_1 \odot e_1 +e_2 \odot e_2 -e_3 \odot e_3 +e_4 \odot e_4),\\
h_6 &= \omega_2 \eta_3 =\tfrac{1}{2}(e_1 \odot e_2 -e_3 \odot e_4) ,\\   
h_7 &= \omega_3 \eta_1 =\tfrac{1}{2}(e_1 \odot e_3 -e_2 \odot e_4),\\   
h_8 &= \omega_3 \eta_2 =\tfrac{1}{2}(-e_1 \odot e_2 -e_3 \odot e_4) ,\\*   
h_9 &= \omega_3 \eta_3 =\tfrac{1}{4}(-e_1 \odot e_1 +e_2 \odot e_2 +e_3 \odot e_3 -e_4 \odot e_4).  
\end{align*}
Relative to this basis, the matrix of $\mathring R$ is
\begin{equation*}
    \mathring R =\begin{pmatrix}
        D_1 & O_1 & O_2 \\
        O_1^t & D_2 & O_3 \\
        O_2^t & O_3^t & D_3
    \end{pmatrix}.
\end{equation*}
The diagonal blocks are
\begin{equation*}
    D_1 = \frac{S}{4}\id + \begin{pmatrix}
    -A_{11}-C_{11} & -C_{12} & -C_{13} \\
     -C_{12} & -A_{11}-C_{22} & -C_{23}  \\
     -C_{13} & -C_{23} & -A_{11}  -C_{33}  
    \end{pmatrix}, 
\end{equation*}
\begin{equation*}
    D_2 = \frac{S}{4}\id + \begin{pmatrix}
     -A_{22}-C_{11} & -C_{12} & -C_{13}  \\
     -C_{12} & -A_{22}-C_{22} & -C_{23}  \\
     -C_{13} & -C_{23} & -A_{22}-C_{33}
    \end{pmatrix},
\end{equation*}
and
\begin{equation*}
    D_3=\frac{S}{4}\id + \begin{pmatrix}
    -A_{33}-C_{11} & -C_{12} & -C_{13} \\
 -C_{12} & -A_{33}-C_{22} & -C_{23} \\
     -C_{13} & -C_{23} & -A_{33}-C_{33} 
    \end{pmatrix}.
\end{equation*}
The off-diagonal blocks are
\begin{equation*}
    O_1 =\begin{pmatrix}
      -A_{12} & B_{33} & -B_{32}  \\
      -B_{33} & -A_{12} & B_{31}  \\
      B_{32} & -B_{31} & -A_{12}
    \end{pmatrix},
\end{equation*}
\begin{equation*}
    O_2 =\begin{pmatrix}
     -A_{13} & -B_{23} & B_{22} \\
      B_{23} & -A_{13} & -B_{21} \\
      -B_{22} & B_{21} & -A_{13}
    \end{pmatrix},
\end{equation*}
and 
\begin{equation*}
    O_3 =\begin{pmatrix}
     -A_{23} & B_{13} & -B_{12} \\
     -B_{13} & -A_{23} & B_{11} \\
      B_{12} & -B_{11} & -A_{23}
    \end{pmatrix}.
\end{equation*}
We verify these matrices using Lemma~\ref{lemma: eq identities}.
For $D_1$, we obtain
\begin{align*}
2\mathring R(h_1, h_1) &=  -R_{1212}-R_{3434}+R_{1313}+R_{1414}+R_{2323}+R_{2424}  =\frac{S}{2}-2A_{11}-2C_{11}, \\
2\mathring R(h_1,h_2) & =  -R_{1213}+R_{1242}+R_{1334}-R_{3442} =-2\widehat R(\eta_1, \eta_2)=-2 C_{12}, \\ 
2\mathring R(h_1,h_3) &=  -R_{1214}+R_{1223}+R_{1434}-R_{2334} =-2\widehat R(\eta_1,\eta_3) =-2C_{13} , \\
2\mathring R(h_2,h_2) &= R_{1414}+R_{2323} -2R_{1234}+2R_{1342} = \frac{S}{2} - 2A_{11}-2C_{22}, \\
2\mathring R(h_2,h_3) & = -R_{1314} +R_{1442}+R_{1323}-R_{2342} = -2C_{23}, \\*
2\mathring R(h_3,h_3) &= R_{1313}+R_{2424} -2R_{1234}+2R_{1423}  = \frac{S}{2} - 2A_{11}-2C_{33}.
\end{align*}
For $D_2$, we obtain
\begin{align*}
2\mathring R(h_4,h_4) &= R_{1414} +R_{2323} +2R_{1234}-2R_{1342}  = \frac{S}{2} - 2A_{22}-2C_{11}, \\
2\mathring R(h_4,h_5) &= R_{1242}+R_{1334} -R_{1213}-R_{3442} = -2\widehat R(\eta_1, \eta_2), \\
2\mathring R(h_4,h_6) &= -R_{1214} +R_{1434}+R_{1223}-R_{2334} = -2\widehat R(\eta_1,\eta_3), \\
2\mathring R(h_5,h_5) &= -R_{1313}-R_{2424} +R_{1414}+R_{2323}+R_{1212}+R_{3434} =\frac{S}{2} -2A_{22} -2C_{22}, \\
2\mathring R(h_5,h_6) &= -R_{1314} -R_{2342}+R_{1323}+R_{1442} = -2\widehat R(\eta_2,\eta_3), \\*
2\mathring R(h_6,h_6) &= R_{1212}+R_{3434}-2R_{1342} +2R_{1423} = \frac{S}{2} - 2A_{22}-2C_{33},
\end{align*}
and for $D_3$, we obtain
\begin{align*}
2\mathring R(h_7,h_7) &= R_{1313}+R_{2424} +2R_{1234}-2R_{1423} = \frac{S}{2} - 2A_{33}-2C_{11}, \\
2\mathring R(h_7,h_8) &= -R_{1213}+R_{1334}+R_{1242}-R_{3442} = -2\widehat R(\eta_1,\eta_2), \\
2\mathring R(h_7,h_9) &= R_{1223}+R_{1434}-R_{1214}-R_{2334} = -2\widehat R(\eta_1, \eta_3), \\
2\mathring R(h_8,h_8) &= R_{1212}+R_{3434}+2R_{1342}-2R_{1423} = \frac{S}{2} - 2A_{33}-2C_{22}, \\
2\mathring R(h_8,h_9) &= R_{1323}+R_{1442}-R_{1314}-R_{2342} = -2\widehat R(\eta_2,\eta_3), \\*
2\mathring R(h_9,h_9) &= -R_{1414}-R_{2323} +R_{1313}+R_{1212}+R_{3434}+R_{2424}  =\frac{S}{2} -2A_{33} -2C_{33}.
\end{align*}
For $O_1$, we compute
\begin{align*}
2\mathring R(h_1,h_4) &=  -R_{1213}-R_{1242}-R_{1334}-R_{3442} =-2A_{12}, \\
2\mathring R(h_1,h_5) &= R_{1414} -R_{2323} =2\widehat R(\omega_3, \eta_3), \\
2\mathring R(h_1,h_6) &=  -R_{1314}+R_{2342}-R_{1323}+R_{1442} =-2\widehat R(\eta_2, \omega_3), \\
2\mathring R(h_2,h_4) &= -R_{1414}+R_{2323}=- 2\widehat R(\omega_3, \eta_3),\\
2\mathring R(h_2,h_5) &= -R_{1242}-R_{1334}-R_{1213}-R_{3442} = -2\widehat R(\omega_1,\omega_2), \\
2\mathring R(h_2,h_6) &= R_{1214}-R_{2334}-R_{1434}+R_{1223} = 2\widehat R(\omega_3,\eta_1), \\
2\mathring R(h_3,h_4) &= R_{1314}+R_{1323} -R_{1442}-R_{2342} = 2\widehat R(\omega_3, \eta_2), \\
2\mathring R(h_3,h_5) &= -R_{1223}+R_{1434}-R_{1214}+R_{2334} = -2\widehat R(\omega_3,\eta_1), \\*
2\mathring R(h_3,h_6) &= -R_{1213}-R_{1334}-R_{1242}-R_{3442} = -2\widehat R(\omega_1, \omega_2). 
\end{align*}
For $O_2$, we compute
\begin{align*}
2\mathring R(h_1,h_7) &=  -R_{1214}-R_{1223}-R_{1434}-R_{2334} =-2\widehat R(\omega_1,\omega_3), \\
2\mathring R(h_1,h_8) &=  -R_{1314}+R_{2342}+R_{1323}-R_{1442} =-2\widehat R(\omega_2,\eta_3), \\
2\mathring R(h_1,h_9) &= R_{1313} -R_{2424} =2\widehat R(\omega_2, \eta_2), \\
2\mathring R(h_2,h_7) &= R_{1314} -R_{2342} +R_{1442}-R_{1323} = 2 \widehat R(\omega_2, \eta_3), \\
2\mathring R(h_2,h_8) &= -R_{1214}-R_{1434}-R_{1223}-R_{2334} = -2 \widehat R(\omega_1, \omega_3), \\ 
2\mathring R(h_2,h_9) &= -R_{1242}+R_{1334}-R_{1213}+R_{3442} = -2\widehat R(\omega_2,\eta_1), \\ 
2\mathring R(h_3,h_7) &= -R_{1313}+R_{2424}=- 2\widehat R(\omega_2, \eta_2),\\
2\mathring R(h_3,h_8) &= R_{1213}-R_{1334}+R_{1242}-R_{3442} = 2 \widehat R(\omega_2, \eta_1), \\*
2\mathring R(h_3,h_9) &= -R_{1223}-R_{1434}-R_{1214}-R_{2334} = -2\widehat R(\omega_1, \omega_3),
\end{align*}
and for $O_3$, we compute
\begin{align*}
2\mathring R(h_4,h_7) &=  -R_{1314}-R_{1442}-R_{1323}-R_{2342}  =-2\widehat R(\omega_2, \omega_3), \\
2\mathring R(h_4,h_8) &=R_{1214}+R_{1434}-R_{1223}-R_{2334} = 2\widehat R(\omega_1, \eta_3), \\
2\mathring R(h_4,h_9) &= R_{1242}-R_{1334}-R_{1213}+R_{3442} = -2\widehat R(\omega_1, \eta_2), \\
2\mathring R(h_5,h_7) &=-R_{1214}+R_{1223}+R_{2334}-R_{1434} = -2\widehat R(\omega_1, \eta_3),  \\
2\mathring R(h_5,h_8) &= -R_{1314}-R_{2342}-R_{1323}-R_{1442} = -2\widehat R(\omega_2,\omega_3), \\
2\mathring R(h_5,h_9) &= R_{1212}-R_{3434} =2\widehat R(\omega_1, \eta_1), \\
2\mathring R(h_6,h_7) &= R_{1213}-R_{3442}-R_{1242}+R_{1334} = 2\widehat R(\omega_1,\eta_2), \\
2\mathring R(h_6,h_8) &= -R_{1212}+R_{3434}=- 2\widehat R(\omega_1, \eta_1),\\* 
2\mathring R(h_6,h_9) &= -R_{1323}-R_{1442}-R_{1314}-R_{2342} = -2\widehat R(\omega_2,\omega_3). 
\end{align*}
These identities, together with \eqref{eq:omega eta commutator}, verify \eqref{eq:main identity basis s3} for all $1\leq i,j,k,l \leq 3$. 
\end{proof}

We now pass from the basis identity to arbitrary two-forms.
\begin{proof}[Proof of Theorem~\ref{thm:main}]
Since $\{\omega_1,\omega_2,\omega_3\}$ and
$\{\eta_1,\eta_2,\eta_3\}$ are oriented orthonormal bases of
$\Lambda^+V$ and $\Lambda^-V$, respectively, 
we can write \(\omega,\xi\in\Lambda^+ V \) and \(\eta,\zeta\in\Lambda^-V\) as
\[
 \omega=x_i\omega_i,\qquad
 \xi=z_k\omega_k,\qquad
 \eta=y_j\eta_j,\qquad
 \zeta=w_l\eta_l.
\]
By bilinearity,
\[
\omega\eta=x_i y_j\omega_i \eta_j,\qquad \xi \zeta =z_kw_l \omega_k \eta_l,
\]
and Proposition~\ref{prop: main identity} gives
\begin{align*}
 \mathring R(\omega\eta,\xi\zeta)
 &=x_iy_jz_kw_l\mathring R(\omega_i\eta_j,\omega_k\eta_l)\\
 &=x_iy_jz_kw_l\Bigl(\frac S4\delta_{ik}\delta_{jl}
 -\delta_{jl}\widehat R(\omega_i,\omega_k)
 -\delta_{ik}\widehat R(\eta_j,\eta_l)
 -\frac12\widehat R([\omega_i,\omega_k],[\eta_j,\eta_l])\Bigr)\\*
 &=\frac S4\langle\omega,\xi\rangle\langle\eta,\zeta\rangle
 -\langle\eta,\zeta\rangle\widehat R(\omega,\xi)
 -\langle\omega,\xi\rangle\widehat R(\eta,\zeta)
 -\frac12\widehat R([\omega,\xi],[\eta,\zeta]),
\end{align*}
which proves the theorem.  
\end{proof}

\begin{proof}[Proof of Corollary~\ref{cor:diagonal-identity}]
Take $\xi=\omega$ and $\zeta=\eta$ in \eqref{eq:master}.
Since $[\omega,\omega]=0$ and $[\eta,\eta]=0$, the commutator term
vanishes, and the remaining terms give
\eqref{eq:diagonal-identity}.
\end{proof}

\section{Proof of Theorem~\ref{thm:pinching}}
\label{sec:pinching-proof}

\begin{proof}[Proof of Theorem~\ref{thm:pinching}]
At each point $p\in M$, put $\alpha=\lambda_{\max}(W^+)$ and
$\beta=\lambda_{\max}(W^-)$. Since $W^\pm$ are trace-free,
$\alpha,\beta\geq0$. Corollary~\ref{cor:diagonal-identity} and
\eqref{eq:biorthogonal-extrema} give
\begin{equation*}
 \min_{h\in\Sigma_p}\mathring R(h,h)
 =\frac S4-2K^\perp_{\max}
 =\frac S{12}-\alpha-\beta.
\end{equation*}
Thus the non-strict hypothesis is equivalent to
\begin{equation}\label{eq:weyl-pinching}
 \alpha+\beta\leq\frac S6,
\end{equation}
and the strict hypothesis gives $\alpha+\beta<S/6$.
In particular, $S\geq0$ and $\alpha,\beta\leq S/6$, with strict
inequalities under the strict hypothesis.

These inequalities imply nonnegative, respectively positive, isotropic
curvature. Indeed, for an oriented orthonormal frame and
$\theta=(e_1\wedge e_2+e_3\wedge e_4)/\sqrt2$, one has
\begin{equation*}
 R_{1313}+R_{1414}+R_{2323}+R_{2424}-2R_{1234}
 =2\left(\frac S6-\langle W^+\theta,\theta\rangle\right),
\end{equation*}
with the analogous formula for $W^-$ after reversing orientation.
The classification of Chen-Tang-Zhu~\cite{CTZ12} now proves
part~(1): orientability excludes their nonorientable connected-sum
factors, and a $\mathbb{S}^4$ summand does not change the diffeomorphism
type.

For part~(2), if $M$ admits a metric with positive isotropic curvature,
the same classification gives~(a). Otherwise, we claim that $M$ has a
finite cover $N$ with $b_2(N)>0$. By
Micallef-Wang~\cite[Theorem~4.10]{MW93}, at a sufficiently small
positive time under the Ricci flow, the remaining possibilities are a flat
metric, a product of two surfaces on the universal cover, or a
K\"ahler metric on a manifold diffeomorphic to $\mathbb{CP}^2$.
In the flat case, take a finite torus cover. In the product case,
the uniqueness of the de Rham decomposition allows us to pass to a
finite cover preserving both surface factors and their orientations;
the parallel factor area forms then descend to nonzero harmonic
forms and hence represent nonzero cohomology classes. In the last
case, $b_2(M)=1$. This proves the claim.

Equip $N$ with the lift of the original metric $g$. By Hodge theory,
after reversing orientation if necessary, there is a nonzero harmonic
self-dual two-form $\varphi$. The Weitzenb\"ock formula gives
\begin{equation}\label{eq:harmonic-bochner}
 0=\int_N\left(|\nabla\varphi|^2+
 \left\langle\left(\frac S3\id-2W^+\right)\varphi,
 \varphi\right\rangle\right)\,dV.
\end{equation}
Since $\alpha\leq S/6$, both terms are nonnegative, so $\varphi$ is
parallel. Normalize it so that $|\varphi|^2=2$. It is then the
K\"ahler form of a parallel orthogonal complex structure.
For this orientation, the K\"ahler
identity
\[
 \operatorname{Spec}(W^+)
 =\left\{\frac S6,-\frac S{12},-\frac S{12}\right\}
\]
gives $\alpha=S/6$. Hence \eqref{eq:weyl-pinching} forces $\beta=0$,
and the trace-free property of $W^-$ gives $W^-=0$.

Thus $N$ is a compact self-dual K\"ahler surface with $S\geq0$.
Derdzi\'nski's theorem~\cite[Theorem~1]{Der83} implies that $N$ is
locally symmetric. By the classification recalled in
Itoh~\cite[Theorem~C]{Ito84},
its universal Riemannian cover, which is also that of $M$, is
$\mathbb R^4$, $\mathbb{S}^2(c)\times\mathbb H^2(-c)$ for some $c>0$,
or $\mathbb{CP}^2$ with a scaled Fubini-Study metric.
The first two possibilities give~(b) and~(c). In the last case,
orient $M$ so that the universal covering is orientation-preserving
for the complex orientation on $\mathbb{CP}^2$. If its degree is $d$,
multiplicativity of the signature gives
\[
 1=\tau(\mathbb{CP}^2)=d\,\tau(M).
\]
Consequently $d=1$, giving~(d).

The alternatives are mutually exclusive. For any metric with positive
isotropic curvature, the curvature term in
\eqref{eq:harmonic-bochner} is positive definite on both
$\Lambda^+$ and $\Lambda^-$, so every finite cover has $b_2=0$.
Each of cases~(b), (c), and~(d) has a finite cover with $b_2>0$ by
the preceding arguments. Thus none can occur together with~(a).
Cases~(b), (c), and~(d) are distinguished by their universal
Riemannian covers.

Finally, if $M$ is simply connected, the noncompact universal covers
in~(b) and~(c) are impossible. In~(a), every compact cylinder quotient
has infinite fundamental group, and the fundamental group of a
connected sum is the free product of those of its summands. Thus
only $\mathbb{S}^4$ remains in~(a), proving the final assertion.
\end{proof}



\bibliographystyle{alpha}
\bibliography{references}

\end{document}